\documentclass[11pt,reqno]{amsart}
\usepackage{amsmath}
\usepackage{amssymb}
\usepackage[bookmarks]{hyperref}

\newtheorem{theorem}{Theorem}[section]
\newtheorem{corollary}[theorem]{Corollary}
\newtheorem{lemma}[theorem]{Lemma}
\newtheorem{proposition}[theorem]{Proposition}
 
\theoremstyle{definition}

\theoremstyle{remark}
\newtheorem{remark}[theorem]{Remark}

\numberwithin{equation}{section}

\newcommand{\opdisk}{\mathbb{D}}
\newcommand{\cldisk}{\overline{\mathbb{D}}}
\newcommand{\bdry}{\partial\opdisk^n}
\newcommand{\polyint}[2]{\mathrm{d}\mu_1({#2}_1)\cdots\mathrm{d}\mu_{#1}({#2}_{#1})}

\allowdisplaybreaks[1]

\providecommand{\bysame}{\leavevmode\hbox to3em{\hrulefill}\thinspace}
\providecommand{\MR}{\relax\ifhmode\unskip\space\fi MR }

\providecommand{\href}[2]{#2}

\begin{document}

\title[Compact Hankel operators]{Compact Hankel operators on\\ generalized Bergman spaces of the polydisc}
\author{Trieu Le}
\address{Department of Mathematics\\
Mail Stop 942\\
University of Toledo\\
Toledo, OH 43606}
\email{trieu.le2@utoledo.edu}

\begin{abstract} Let $\vartheta$ be a measure on the polydisc $\opdisk^n$ which is the product of $n$ regular Borel probability measures so that $\vartheta([r,1)^n\times\mathbb{T}^n)>0$ for all $0<r<1$. The Bergman space $A^2_{\vartheta}$ consists of all holomorphic functions that are square integrable with respect to $\vartheta$. In one dimension, it is well known that if $f$ is continuous on the closed disc $\cldisk$, then the Hankel operator $H_{f}$ is compact on $A^2_{\vartheta}$. In this paper we show that for $n\geq 2$ and $f$ a continuous function on $\cldisk^n$, $H_{f}$ is compact on $A^2_{\vartheta}$ if and only if there is a decomposition $f=h+g$, where $h$ belongs to $A^2_{\vartheta}$ and $\lim_{z\to\bdry}g(z)=0$.
\end{abstract}
\maketitle

\section{Introduction}\label{S:intro}

Fix a positive integer $n\geq 1$. Let $\mathbb{D}^n$ be the open unit polydisc in $\mathbb{C}^n$ 
and let $\mathbb{T}^n$ be the $n$-torus, which is the Shilov boundary of $\mathbb{D}^n$.  
The closure of $\mathbb{D}^n$ is $\cldisk^n$, the product of $n$ copies of the closed unit disc. 

For $z=(z_1,\ldots,z_n)\in\mathbb{C}^n$ and $\zeta=(\zeta_1,\ldots,\zeta_n)\in\mathbb{T}^n$, 
we use $z\cdot\zeta$ and $\zeta\cdot z$ to denote the point $(z_1\zeta_1,\ldots,z_n\zeta_n)$. We write $\overline{z}=(\overline{z}_1,\ldots,\overline{z}_n)$, 
and for any $m=(m_1,\ldots,m_n)$ in $\mathbb{Z}^n$, we write $z^m=z_1^{m_1}\cdots z_n^{m_n}$ whenever it is defined. 
We use $\sigma$ to denote the surface measure on $\mathbb{T}^n$ which is normalized so that $\sigma(\mathbb{T}^n)=1$. 
Let $\mu$ be a regular Borel probability measure on $[0,1)^n$. Then there is a regular Borel probability measure 
$\vartheta$ on $\opdisk^n$ so that 
\begin{equation}\label{Eq:int-polar}
\int_{\opdisk^n}f(z)\mathrm{d}\vartheta(z) 
= \int_{[0,1)^n}\Big\{\int_{\mathbb{T}^n}f(r\cdot\zeta)\mathrm{d}\sigma(\zeta)\Big\}\mathrm{d}\mu(r)
\end{equation}
for all continuous functions $f$ with compact support on $\opdisk^n$. 
It then follows that the above identity also holds true for all functions $f$ in $L^1(\opdisk^n,\vartheta)$. 

In this paper we are interested in those measures $\mu$ which satisfy the condition $\mu([r,1)^n)>0$ for $0<r<1$. 
This implies that $\vartheta(\{z\in\opdisk^n: |z_1|\geq r,\ldots,|z_n|\geq r\})>0$ for $0<r<1$. 
We write $L^2_{\vartheta}$ for $L^2(\opdisk^n,\vartheta)$ and $\|\cdot\|_{2}$ for the norm on $L^2_{\vartheta}$. 
The Bergman space $A^2_{\vartheta}$ is the closure in $L^2_{\vartheta}$ of the space of all holomorphic polynomials. 
The condition on $\mu$ will imply that all functions in $A^2_{\vartheta}$ are holomorphic on the polydisc. 
Let $P$ denote the orthogonal projection from $L^2_{\vartheta}$ onto $A^2_{\vartheta}$. 
For any function $f$ in $L^2_{\vartheta}$, the (big) Hankel operator $H_{f}$ is a densely defined operator 
from $A^2_{\vartheta}$ into $L^2_{\vartheta}\ominus A^2_{\vartheta}$ by $H_{f}(\varphi) = (1-P)(f\varphi)$ 
for all holomorphic polynomials $\varphi$. The function $f$ will be called a symbol of the operator $H_f$. 
It is clear that if $f$ is bounded, then $H_{f}$ is a bounded operator with $\|H_{f}\|\leq\|f\|_{\infty}$. 
However, there are unbounded functions $f$ such that $H_{f}$ extends to a bounded operator on $A^2_{\vartheta}$. 
In fact, if $f$ belongs to $A^2_{\vartheta}$, then since $f\varphi$ belongs to $A^2_{\vartheta}$ 
for all holomorphic polynomials $\varphi$, we conclude that $H_{f}=0$. 
Conversely, if $H_{f}=0$, then since $0 = H_{f}(1)=(I-P)(f)$, we see that $f$ must belong to $A^2_{\vartheta}$. 
Therefore, $H_{f}=0$ if and only if $f$ is in $A^2_{\vartheta}$. 
This shows that a Hankel operator has many symbols and any two symbols of the same operator differ by a function in $A^2_{\vartheta}$.

It is well known that if a function $g\in L^2_{\vartheta}$ vanishes outside a compact subset of $\mathbb{D}^n$, 
then the Hankel operator $H_{g}$ is compact. Let $\bdry$ be the topological boundary of $\opdisk^n$ as a subset of $\mathbb{C}^n$. 
If $g\in L^2_{\vartheta}$ such that $\lim_{z\to\bdry}g(z)=0$ (that is, for any $\epsilon>0$, 
there is a compact subset $M$ of $\opdisk^n$ so that $|g(z)|<\epsilon$ for $\vartheta$-a.e. $z$ in $\opdisk^n\backslash M$), 
then an approximation argument shows that $H_{g}$ is also a compact operator. 
This together with the above fact about zero Hankel operators implies that if $f=h+g$, 
where $h$ belongs to $A^2_{\vartheta}$ and $\lim_{z\to\bdry}g(z)=0$, then $H_{f}$ is compact.

In the one-dimensional case ($n=1$), it is well known that if $f$ is continuous on $\overline{\mathbb{D}}$, 
then $H_{f}$ is compact. See \cite[p. 226]{ZhuAMS2007} for the case of weighted Bergman spaces. 
For generalized Bergman spaces, one can prove this by checking directly that $H_{\bar{z}^{i} z^{j}}$ 
is compact for all integers $i, j\geq 0$. See Section \ref{S:Hankel-quasi} for more details. 
The case $n\geq 2$ turns out to be completely different. 
Not all Hankel operators with continuous symbols are compact. 
More surprisingly, we will show, under the assumption that $\mu$ is the product of $n$ measures on $[0,1)$, 
that if $f$ is continuous on $\overline{\mathbb{D}}^n$, 
then $H_{f}$ is compact if and only if $f$ admits a decomposition $f=h+g$, 
where $h$ belongs to $A^2_{\vartheta}$ and $\lim_{z\to\bdry}g(z)=0$.

If $\mathrm{d}\mu(r_1,\ldots,r_n)=2^nr_1\cdots r_n\mathrm{d}r_1\cdots\mathrm{d}r_n$, then
$A^2_{\vartheta}$ is the usual Bergman space of the polydisc. 
K. Stroethoff \cite{StroethoffIJM1990,StroethoffJOT1990} and D. Zheng \cite{ZhengIEOT1989} gave necessary 
and sufficient conditions on a bounded function $f$ for which $H_{f}$ is compact. 
However, their conditions, which involve the projection $P$ and Mobius transformations, 
are difficult to check. Indeed, even if a function $f$ is assumed to be continuous on $\overline{\mathbb{D}}^n$, 
it is not clear from their results what geometric conditions $f$ needs to satisfy in order for $H_f$ to be compact. 
Our approach (though works only for continuous functions) is different from theirs and our result is more transparent.

To conclude the section, we would like to mention some results on the compactness of Hankel operators 
on the Hardy space $H^2=H^2(\mathbb{T}^n)$. 
In the one-dimensional case, it is a classical theorem of Hartman (see \cite[Chapter 10]{ZhuAMS2007}) 
that $H_{f}$ can be extended to a compact operator if and only if $f=h+g$, 
where $h$ belongs to $H^2$ and $g$ is continuous on the circle $\mathbb{T}$. 
On the other hand, the case $n\geq 2$ is much different. 
It was showed by M. Cotlar and C. Sadosky in \cite{CotlarIEOT1993} 
and P. Ahern, E.H. Youssfi and K. Zhu in \cite{AhernJOT2009} with a different approach 
that if $H_f$ is compact, then $f$ must belong to $H^2$. 
This means that there is no non-zero compact Hankel operator on $H^2$. 
This result was also proved in the setting of Hardy-Sobolev spaces on the polydisc 
by Ahern, Youssfi and Zhu in the same paper with the same approach. 
Our analysis in the present paper was actually motivated by their results and techniques.

\section{Preliminaries}\label{S:prel}

In this section we explain in more details some of the results that we mentioned in the Introduction. 
From Cauchy's formula and the assumption that $\mu([r,1)^n)>0$ for all $0<r<1$, 
for any compact subset $M$ of $\mathbb{D}^n$, there is a positive constant $C_{M}$ 
so that $|p(z)|\leq C_M\|p\|_{2}$ for all $z\in M$, and all holomorphic polynomials $p$. 
This implies that for $f\in A^2_{\vartheta}$, $f$ is holomorphic on $\mathbb{D}^n$ 
and we also have $|f(z)|\leq C_M\|f\|_{2}$ for all $z\in M$. 
In fact, it can be showed that $A^2_{\vartheta}$ is the space of all functions 
in $L^2_{\vartheta}$ that are holomorphic on $\mathbb{D}^n$. 
Since $|f(z)|\leq C_M\|f\|_{2}$, the valuation map $z\mapsto f(z)$ is a continuous functional 
on $A^2_{\vartheta}$ for any $z\in\mathbb{D}^n$. 
So there is a function $K_z$ in $A^2_{\vartheta}$ such that $f(z)=\langle f,K_z\rangle$ 
for any $f\in A^2_{\vartheta}$. The function $K_z$ is called the reproducing kernel at $z$. 
For any compact subset $M$ and for any $z\in M$, since $K_z(z)\leq C_M\|K_z\|_{2}=C_M(K_z(z))^{1/2}$, we have $K_z(z)\leq C_M^2$.

From \eqref{Eq:int-polar}, the monomials $\{z^m: m\in\mathbb{Z}_{+}^n\}$ are pairwise orthogonal. 
On the other hand, the linear span of these monomials is dense in $A^2_{\vartheta}$. 
Therefore $A^2_{\vartheta}$ has the following orthonormal basis, usually referred to as the standard orthonormal basis, $\{e_{m}(z)=\frac{z^m}{\sqrt{c_{m}}}: m\in\mathbb{Z}_{+}^n\}$, where $$c_{m}=\int_{\opdisk^n}z^m\bar{z}^m\mathrm{d}\vartheta(z)=\int_{[0,1)^n}r_1^{2m_1}\cdots r_n^{2m_n}\mathrm{d}\mu(r_1,\ldots,r_n).$$

Suppose $f$ is a function in $L^2_{\vartheta}$. Then
\begin{align}\label{E:Hilbert-Schmidt}
\sum_{m\in\mathbb{Z}_{+}^n}\|H_fe_{m}\|_{2}^2 & \leq \sum_{m\in\mathbb{Z}_{+}^n}\|fe_{m}\|_{2}^2\notag\\
& = \int_{\opdisk^n}|f(z)|^2\Big\{\sum_{m\in\mathbb{Z}_{+}^n}|e_{m}(z)|^2\Big\}\mathrm{d}\vartheta(z)\\
& = \int_{\opdisk^n}|f(z)|^2K_z(z)\mathrm{d}\vartheta(z),\notag
\end{align}
where the last equality follows from the well known formula $$K_z(z)=\|K_z\|_{2}^2 = \sum_{m\in\mathbb{Z}_{+}^n}|\langle K_{z},e_{m}\rangle|^2=\sum_{m\in\mathbb{Z}_{+}^n}|e_{m}(z)|^2.$$
If $f$ vanishes outside a compact subset $M$ of $\mathbb{D}^n$, then \eqref{E:Hilbert-Schmidt} gives
\begin{align*}
\sum_{m\in\mathbb{Z}_{+}^n}\|H_fe_{m}\|_{2}^2 &\leq C^2_{M}\int_{M}|f(z)|^2\mathrm{d}\vartheta(z)<\infty,
\end{align*}
since $K_z(z)\leq C^2_M<\infty$ for all $z\in M$. Thus, $H_{f}$ is a Hilbert-Schmidt operator, hence it is compact. 

Suppose $f$ belongs to $L^2_{\vartheta}$ so that $\lim_{z\to\bdry}f(z)=0$. 
Then for any $\epsilon>0$, there is a compact subset $M_{\epsilon}\subset\mathbb{D}^n$ so that $|f(z)|<\epsilon$ for $\vartheta$-a.e. $z\in\mathbb{D}^n\backslash M_{\epsilon}$. 
This implies that $\|f-f\chi_{M_{\epsilon}}\|_{\infty}\leq\epsilon$. 
And hence, $\|H_{f}-H_{f\chi_{M_{\epsilon}}}\|\leq\epsilon$. As we have seen above, 
$H_{f\chi_{M_{\epsilon}}}$ is a compact operator for each $\epsilon$. 
Therefore, $H_{f}$, being the limit of a net of compact operators, is also a compact operator. 
Thus we have showed the following well known result.

\begin{proposition}\label{P:cptHankel} 
Suppose $f=h+g$, where $h\in A^2_{\vartheta}$ and $g\in L^2_{\vartheta}$ so that $\lim_{z\to\bdry}g(z)=0$. Then $H_{f}$ is a compact operator on $A^2_{\vartheta}$.
\end{proposition}

In the rest of the section, we study a decomposition of $L^2_{\vartheta}$ into pairwise orthogonal subspaces. If a function belongs to one of these subspaces, the corresponding Hankel operator has a simple form which we can analyze easily. This is one of the key points in our study of the compactness of Hankel operators with continuous symbols.

For any $n$-tuple $l\in\mathbb{Z}^n$, let $\mathcal{H}_{l}$ be the space of all functions $f$ in $L^2_{\vartheta}$ such that for all $\zeta\in\mathbb{T}^n$, $f(\zeta\cdot z)=\zeta^{l}f(z)$ for $\vartheta$-a.e. $z\in\opdisk^n$. Following \cite{LouhichiIEOT2006}, we call each function in $\mathcal{H}_{l}$ quasi-homogeneous of multi-degree $l$. It is clear that $\mathcal{H}_{l}$ is a closed subspace of $L^2_{\vartheta}$. Let $Q_l$ denote the orthogonal projection from $L^2_{\vartheta}$ onto $\mathcal{H}_l$. The following lemma shows that these projections are pairwise orthogonal and they constitute a partition of the identity.

\begin{lemma}\label{L:decomposition}
For $s\in\mathbb{Z}^n$ and $f\in L^2_{\vartheta}$, we have
\begin{equation}\label{Eq:Q_s}
(Q_{s}f)(z)=\int_{\mathbb{T}^n}f(z\cdot\zeta)\bar{\zeta}^{s}\mathrm{d}\sigma(\zeta),
\end{equation}
for $\vartheta$-a.e. $z\in\opdisk^n$. Furthermore, $\mathcal{H}_{l}\bot\mathcal{H}_{s}$ (which implies $Q_lQ_s=0$) whenever $l\neq s$, and $L^2_{\vartheta}=\bigoplus_{s\in\mathbb{Z}^n}Q_s(L^2_{\vartheta})=\bigoplus_{s\in\mathbb{Z}^n}\mathcal{H}_{s}$.
\end{lemma}

\begin{proof}
Since $f$ belongs to $L^2_{\vartheta}$, the integral on the right hand side of \eqref{Eq:Q_s} is well-defined for $\vartheta$-a.e.  $z\in\opdisk^n$. For such $z$, let $f_s(z)$ be the value of the integral. For other values of $z$, let $f_{s}(z)=0$. We will show $Q_sf=f_s$ by proving that $f_s\in\mathcal{H}_s$ and $(f-f_s)\bot\mathcal{H}_s$. For $z$ and any $\gamma\in\mathbb{T}^n$, if the integral in \eqref{Eq:Q_s} is defined, by the rotation invariance of $\sigma$, we have
\begin{align*}
f_s(z\cdot\gamma) & = \int_{\mathbb{T}^n}f((z\cdot\gamma)\cdot\zeta)\bar{\zeta}^{s}\mathrm{d}\sigma(\zeta) = \int_{\mathbb{T}^n}f(z\cdot(\gamma\cdot\zeta))\bar{\zeta}^{s}\mathrm{d}\sigma(\zeta)\\
& = \int_{\mathbb{T}^n}f(z\cdot\zeta)\gamma^{s}\mathrm{d}\sigma(\zeta) = \gamma^{s}f_s(z).
\end{align*}
If the integral in \eqref{Eq:Q_s} is not defined, then $f_s(z\cdot\gamma)=\gamma^{s}f_s(z)$ because they are both zero. Therefore, $f_s$ belongs to $\mathcal{H}_s$.

Now suppose $g$ is a function in $\mathcal{H}_s$. Then
\begin{align*}
\int_{\opdisk^n}f_s(z)\bar{g}(z)\mathrm{d}\vartheta(z)
& = \int_{\opdisk^n}\int_{\mathbb{T}^n}f(z\cdot\zeta)\bar{\zeta}^{s}\bar{g}(z) \mathrm{d}\sigma(\zeta) \mathrm{d}\vartheta(z)\\
& = \int_{\mathbb{T}^n}\int_{\opdisk^n}f(z\cdot\zeta)\bar{\zeta}^{s}\bar{g}(z) \mathrm{d}\vartheta(z)\mathrm{d}\sigma(\zeta)\\
& = \int_{\mathbb{T}^n}\int_{\opdisk^n}f(z\cdot\zeta)\bar{g}(z\cdot\zeta)\mathrm{d}\vartheta(z)\mathrm{d}\sigma(\zeta)\\
&\quad\text{(since $g(z\cdot\zeta)=\zeta^{s}g(z)$ for $\vartheta$-a.e. $z$)}\\
& = \int_{\mathbb{T}^n}\int_{\opdisk^n}f(z)\bar{g}(z)\mathrm{d}\vartheta(z)\mathrm{d}\sigma(\zeta)\\
& = \int_{\opdisk^n}f(z)\bar{g}(z)\mathrm{d}\vartheta(z).
\end{align*}
This shows that $\langle f-f_s,g\rangle=0$ for all $g\in\mathcal{H}_s$. Since $f_s$ belongs to $\mathcal{H}_{s}$, we conclude that $f_s=Q_{s}f$.

Next, suppose $l\neq k$. Let $f\in\mathcal{H}_l$ and $g\in\mathcal{H}_k$. For any $\zeta\in\mathbb{T}^n$, we have
\begin{align*}
\int_{\opdisk^n}\zeta^{l-k} f(z)\bar{g}(z)\mathrm{d}\vartheta(z) & = \int_{\opdisk^n}f(z\cdot\zeta)\bar{g}(z\cdot\zeta)\mathrm{d}\vartheta(z)\\
& = \int_{\opdisk^n}f(z)\bar{g}(z)\mathrm{d}\vartheta(z).
\end{align*}
Since $l\neq k$, we conclude that $\displaystyle\int_{\opdisk^n}f(z)\bar{g}(z)\mathrm{d}\vartheta(z)=0$. Thus, $\mathcal{H}_l\bot\mathcal{H}_k$.

To show $L^2_{\vartheta}=\bigoplus_{l\in\mathbb{Z}^n}\mathcal{H}_{l}$, it suffices to show that for any $f\in L^2_{\vartheta}$, the identity $\displaystyle\|f\|_{2}^2 = \sum_{l\in\mathbb{Z}^n}\|Q_l(f)\|_{2}^2$ holds true. Indeed, for $f\in L^2_{\vartheta}$,
\begin{align*}
\|f\|_{2}^2 & = \int_{\opdisk^n}|f(z)|^2\mathrm{d}\vartheta(z)\\
& = \int_{\opdisk^n}\int_{\mathbb{T}^n}|f(z\cdot\zeta)|^2 \mathrm{d}\sigma(\zeta)\mathrm{d}\vartheta(z)\\
& = \int_{\opdisk^n}\sum_{l\in\mathbb{Z}^n}\Big|\int_{\mathbb{T}^n}f(z\cdot\zeta)\bar{\zeta}^l \mathrm{d}\sigma(\zeta)\Big|^2 \mathrm{d}\vartheta(z)\\
&\quad\text{(since $\{\zeta^l: l\in\mathbb{Z}^n\}$ is an orthonormal basis for $L^2(\mathbb{T}^n,\sigma)$)}\\
& = \sum_{l\in\mathbb{Z}^n}\int_{\opdisk^n}\Big|\int_{\mathbb{T}^n}f(z\cdot\zeta)\bar{\zeta}^l \mathrm{d}\sigma(\zeta)\Big|^2 \mathrm{d}\vartheta(z)\\
& = \sum_{l\in\mathbb{Z}^n}\|Q_l(f)\|_{2}^2.\qedhere
\end{align*}
\end{proof}

It follows from the proof of Lemma \ref{L:decomposition} that for each $s\in\mathbb{Z}^n$, there is a function $f_s$ such that $f_s(z\cdot\gamma)=\gamma^{s}f_s(z)$ for \emph{all} $z\in\opdisk^n$ and \emph{all} $\zeta\in\mathbb{T}^n$ and $Q_{s}(f)(z)=f_s(z)$ for $\vartheta$-a.e. $z$. If $f$ is continuous on the closed polydisc $\overline{\mathbb{D}}^n$, then the integral in \eqref{Eq:Q_s} is well-defined for all $z$ in $\cldisk^n$ and $f_s$ is also continuous on $\cldisk^n$. We have seen that the series $\sum_{s\in\mathbb{Z}^n}f_s$ converges to $f$ in the $L^2_{\vartheta}$-norm. In general, for $f$ in $C(\cldisk^n)$, the series does not converge uniformly to $f$. However, the Ces{\`a}ro means of the functions $\{f_{s}: s\in\mathbb{Z}^n\}$ do converge uniformly to $f$ as we will see next.

For any integer $N\geq 1$, the Ces{\`a}ro mean $\Lambda_{N}(f)$ is defined by the formula
\begin{align*}
&\Lambda_{N}(f)(z)\\
& = \sum_{|s_1|,\ldots,|s_n|\leq N}\big(1-\frac{|s_1|}{N+1}\big)\cdots\big(1-\frac{|s_n|}{N+1}\big)f_{s_1,\ldots,s_n}(z)\\
& = \int_{\mathbb{T}^n}\Big\{\sum_{|s_1|,\ldots,|s_n|\leq N}\big(1-\frac{|s_1|}{N+1}\big)\cdots\big(1-\frac{|s_n|}{N+1}\big)\zeta_1^{s_1}\cdots\zeta_n^{s_1}\Big\}f(z\cdot\zeta)\mathrm{d}\sigma(\zeta)\\
& = \int_{\mathbb{T}^n}\mathbf{F}_{N}(\zeta_1)\cdots\mathbf{F}_{N}(\zeta_n)f(z\cdot\zeta)\mathrm{d}\sigma(\zeta),
\end{align*}
where $\mathbf{F}_{N}$ is the $N$th Fej{\'e}r's kernel. It now follows from a well known result in harmonic analysis (see, for example, Sections 2.2 and 9.2 in \cite{KatznelsonCUP2004}) that $\Lambda_{N}(f)\to f$ uniformly on $\cldisk^n$ as $N\to\infty$ if $f$ is continuous on $\cldisk^n$.

\section{Hankel operators with quasi-homogeneous symbols}\label{S:Hankel-quasi}

Recall from Section \ref{S:prel} that $A^2_{\vartheta}$ has the standard orthonormal basis consisting of monomials $\{e_{m}(z)=\frac{z^m}{\sqrt{c_m}}: m\in\mathbb{Z}_{+}^n\}$, where $$c_m = \int_{[0,1)^n}r_1^{2m_1}\cdots r_n^{2m_n}\mathrm{d}\mu(r_1,\ldots,r_n).$$
We also recall that for $l\in\mathbb{Z}^n$, $Q_{l}$ is the orthogonal projection from $L^2_{\vartheta}$ onto the subspace $\mathcal{H}_{l}$ of quasi-homogeneous functions of multi-degree $l$.

For two $n$-tuples of integers $l=(l_1,\ldots,l_n)$ and $s=(s_1,\ldots,s_n)$, we write $l\succeq s$ if $l_j\geq s_j$ for all $1\leq j\leq n$ and $l\nsucceq s$ if otherwise. We will also use $0$ to denote $(0,\ldots,0)$. For $m\in\mathbb{Z}_{+}^{n}$ and $l\in\mathbb{Z}^n$, $Q_{l}(e_m)$ is either $0$ (when $l\neq m)$ or $e_{m}$ (when $l=m$). Thus, $Q_{l}(A^2_{\vartheta})=\{0\}$ if $l\nsucceq 0$ and $Q_{l}(A^2_{\vartheta})=\mathbb{C}e_l$ if $l\succeq 0$. This shows that $A^2_{\vartheta}$ is an invariant subspace for $Q_{l}$, hence it is also a reducing subspace since $Q_l$ is a projection. Let $P$ be the orthogonal projection from $L^2$ onto $A^2_{\vartheta}$.  Then we have $PQ_l=Q_lP$ and this in turn shows that $\mathcal{H}_{l}$ is a reducing subspace for $P$.

\begin{lemma}\label{L:diagHankel}
Let $s$ be in $\mathbb{Z}^n$. Suppose $f$ is a bounded function on $\opdisk^n$ so that we have $f(r_1\zeta_1,\ldots,r_n\zeta_n)=\zeta^{s}f(r_1,\ldots,r_n)$ for all $z=(r_1\zeta_1,\ldots,r_n\zeta_n)$ in $\opdisk^n$. Then $H^{*}_{f}H_{f}$ is a diagonal operator with respect to the standard orthonormal basis. The eigenvalues of $H^{*}_{f}H_{f}$ are given by
\begin{align*}
\lambda_{m} & = \frac{1}{c_{m}}\int_{[0,1)^n}|f(t_1,\ldots,t_n)|^2t_1^{2m_1}\cdots t_n^{2m_n}\mathrm{d}\mu(t_1,\ldots,t_n)
\end{align*}
if $m+s\nsucceq 0$ and
\begin{align*}
\lambda_{m} & = \frac{1}{c_{m}}\int_{[0,1)^n}|f(t_1,\ldots,t_n)|^2t_1^{2m_1}\cdots t_n^{2m_n}\mathrm{d}\mu(t_1,\ldots,t_n)\\
& \quad\ - \frac{1}{c_{m}c_{m+s}}\Big|\int_{[0,1)^n}f(t_1,\ldots,t_n)t_1^{2m_1+s_1}\cdots t_n^{2m_n+s_n}\mathrm{d}\mu(t_1,\ldots,t_n)\Big|^2
\end{align*}
if $m+s\succeq 0$.
\end{lemma}

\begin{proof}
For any $m\in\mathbb{Z}_{+}^n$, $fe_m$ belongs to $\mathcal{H}_{s+m}$, which is an invariant subspace for $P$. Therefore, $P(fe_m)$ and $H_{f}e_m = fe_m-P(fe_m)$ also belong to $\mathcal{H}_{s+m}$. We have
\begin{align*}
P(fe_m) & = \sum_{k\in\mathbb{Z}_{+}^n}\langle P(fe_m),e_k\rangle e_{k} = \begin{cases}
0 & \text{ if } s+m\nsucceq 0,\\
\langle fe_m,e_{s+m}\rangle e_{s+m} & \text{ if } s+m\succeq 0.
\end{cases}
\end{align*}
Now for $k\neq m$, $\langle H^{*}_{f}H_{f}e_m,e_k\rangle = \langle H_{f}e_m,H_{f}e_k\rangle=0$ since $H_{f}e_m\in\mathcal{H}_{m+s}$, $H_{f}e_{k}\in\mathcal{H}_{k+s}$ and $\mathcal{H}_{m+s}\bot\mathcal{H}_{k+s}$ by Lemma \ref{L:decomposition}. Thus, $H^{*}_{f}H_{f}$ is a diagonal operator with respect to the standard orthonormal basis $\{e_m: m\in\mathbb{Z}_{+}^n\}$. The eigenvalues of $H^{*}_{f}H_{f}$ are given by
\begin{align*}
\lambda_{m} & = \langle H^{*}_{f}H_{f}e_m,e_m\rangle = \|H_{f}e_m\|_{2}^2\\
& = \|fe_m\|_{2}^2 - \|P(fe_m)\|_{2}^2\\
& = \begin{cases}
\|fe_m\|_{2}^2 & \text{ if } s+m\nsucceq 0,\\
\|fe_m\|_{2}^2 - |\langle fe_m,e_{s+m}\rangle|^2 & \text{ if } s+m\succeq 0,
\end{cases}
\end{align*}
for $m\in\mathbb{Z}_{+}^n$. Since $$\|fe_m\|_{2}^2 = \frac{1}{c_m}\int_{[0,1)^n}|f(t_1,\ldots,t_n)|^2t_1^{2m_1}\cdots t_n^{2m_n}\mathrm{d}\mu(t_1,\ldots,t_n),$$ and $$\langle fe_m,e_{s+m}\rangle=\frac{1}{\sqrt{c_m c_{s+m}}}\int_{[0,1)^n}f(t_1,\ldots,t_n)t_1^{2m_1+s_1}\cdots t_n^{2m_n+s_n}\mathrm{d}\mu(t_1,\ldots,t_n),$$
the conclusion of the lemma follows.
\end{proof}

\begin{remark}\label{R:diagHankel}
Let us consider the case $n=1$ and $f(z)=z^u\bar{z}^v$ for integers $u,v\geq 0$. We see that $f$ belongs to $\mathcal{H}_{s}$ with $s=u-v$. From Lemma \ref{L:diagHankel}, $H^{*}_{f}H_{f}$ is a diagonal operator with eigenvalues $\lambda_{m}$ for $m\in\mathbb{Z}_{+}$. For all positive integers $m\geq v-u$, we have
\begin{align*}
\lambda_{m} = \frac{\int_{[0,1)}t^{2m+2(u+v)}\mathrm{d}\mu(t)}{\int_{[0,1)}t^{2m}\mathrm{d}\mu(t)} - \dfrac{\big|\int_{[0,1)}t^{2m+2u}\mathrm{d}\mu(t)\big|^2}{(\int_{[0,1)}t^{2m}\mathrm{d}\mu(t))(\int_{[0,1)}t^{2m+2(u-v)}\mathrm{d}\mu(t))}.
\end{align*}
Since $\mu([r,1))>0$ for all $0<r<1$, it can be showed that $\lim_{m\to\infty}\lambda_{m}=0$ (see Lemma \ref{L:limit} below). Therefore, $H^{*}_{f}H_{f}$ is a compact operator, which implies that $H_{f}$ is also a compact operator. Thus, for any polynomial $p=p(z,\bar{z})$, $H_{p}$ is compact. Since any function $g$ in $C(\cldisk^n)$ can be uniformly approximated by polynomials, we conclude that $H_{g}$ is also a compact operator.
\end{remark}

Our characterization of compactness of Hankel operators when $n\geq 2$ depends on the following lemma. For a sketch of its proof when $\delta=\beta=(0,\ldots,0)$, see Lemma 2.4 in \cite{LePAMS2009}. The proof for arbitrary $\delta, \beta$ is similar.

\begin{lemma}\label{L:limit}
Let $\mu_1,\ldots,\mu_N$ be positive measures on $[0,1)$ so that $\mu_j([r,1))>0$ for all $0<r<1$, all $1\leq j\leq N$. Suppose $\varphi$ is a function on $[0,1)^N$ so that $\lim\limits_{(r_1,\ldots,r_N)\to (1,\ldots,1)}\varphi(r_1,\ldots,r_N)=\alpha$. Then for any $N$-tuples of real numbers $\delta=(\delta_1,\ldots,\delta_N)$ and $\beta=(\beta_1,\ldots,\beta_N)$, we have
$$\lim\limits_{(m_1,\ldots,m_N)\to(\infty,\ldots,\infty)}\dfrac{\int_{[0,1)^N}\varphi(r)r_1^{m_1+\delta_1}\cdots r_N^{m_N+\delta_N}\mathrm{d}\mu(r)}{\int_{[0,1)^N}r_1^{m_1+\beta_1}\cdots r_N^{m_N+\beta_N}\mathrm{d}\mu(r)}=\alpha.$$
\end{lemma}

In the rest of the paper, we will consider only measures $\mu$ of the form $\mathrm{d}\mu(r_1,\ldots,r_n)=\mathrm{d}\mu_1(r_1)\cdots\mathrm{d}\mu_n(r_n)$, where $\mu_1,\ldots,\mu_n$ are regular Borel probability measures on the interval $[0,1)$ such that $\mu_j([r,1))>0$ for all $0<r<1$ and $1\leq j\leq n$. Recall that $\vartheta$ is the measure on $\mathbb{D}^n$ that is related to $\mu$ by equation \eqref{Eq:int-polar}. We now define a measure $\gamma$ on the topological boundary $\partial\mathbb{D}^n$ associated with $\vartheta$. It is clear that $\partial\mathbb{D}^n$ is the disjoint union of $2^n-1$ parts of the form $A_1\times\cdots\times A_n$, where $A_j$ is either $\mathbb{T}$ or $\mathbb{D}$ and not all are $\mathbb{D}$. Suppose $W=A_1\times\cdots\times A_n$ is a part of $\partial\mathbb{D}^n$. Let $1\leq j_1<\cdots<j_p\leq n$ be integers so that $A_{j_1}=\cdots=A_{j_p}=\mathbb{D}$ and $A_j=\mathbb{T}$ if $j\notin\{j_1,\ldots,j_p\}$. We define $\gamma$ to be the unique regular Borel measure on $W$ that satisfies 
\begin{align*}
& \int_{W}f(w)\mathrm{d}\gamma(w)\\
& = \int_{[0,1)^{p}}\Big\{\int_{\mathbb{T}^n}f(\zeta_1,\ldots,r_{j_1}\zeta_{j_1}\ldots,r_{j_{p}}\zeta_{j_p},\ldots,\zeta_n)\mathrm{d}\sigma(\zeta)\Big\}\mathrm{d}\mu_{j_1}(r_{j_1})\cdots\mathrm{d}\mu_{j_p}(r_{j_p})
\end{align*}
for all $f\in C_{c}(W)$. 

The following theorem characterizes compact Hankel operators with continuous quasi-homogeneous symbols when $n\geq 2$.

\begin{theorem}\label{T:cpt-quasi}
Suppose $n\geq 2$ and $s\in\mathbb{Z}^n$. Suppose $f\in\mathcal{H}_{s}$ is  continuous on $\cldisk^n$ such that $H_{f}$ is compact. If $s\succeq 0$, then $f(z)=f(1,\ldots,1)z^s$ for $\gamma$-a.e. $z$ in $\partial\opdisk^n$. If $s\nsucceq 0$, then $f(z)=0$ for $\gamma$-a.e. $z$ in $\partial{\opdisk^n}$.
\end{theorem}

\begin{proof}
By the remark after the proof of Lemma \ref{L:decomposition}, we may assume that $f(\zeta_1 z_1,\ldots,\zeta_n z_n)=\zeta^{s}f(z_1,\ldots,z_n)$ for all $\zeta\in\mathbb{T}^n$ and all $z\in\cldisk^n$. In particular, for all $z=(r_1\zeta_1,\ldots,r_n\zeta_n)$ in $\cldisk^n$, we have 
\begin{equation}\label{Eq:quasi-homo}
f(r_1\zeta_1,\ldots,r_n\zeta_n)=\zeta^{s}f(r_1,\ldots,r_n).
\end{equation}

Suppose $1\leq j<n$. Put $m_1=\max\{0,-s_1\},\ldots, m_j=\max\{0,-s_j\}$. Since $H_{f}$ is compact, $H^{*}_{f}H_{f}$ is also compact. By Lemma \ref{L:diagHankel}, $H_{f}^{*}H_{f}$ is diagonalizable and its eigenvalues are $\lambda_{m}$'s for $m=(m_1,\ldots,m_j,m_{j+1},\ldots,m_n)$. Therefore, $\lim_{(m_{j+1},\ldots,m_{n})\to (\infty,\ldots,\infty)}\lambda_{m}=0$. Using the formula for $\lambda_m$ (when $m+s\succeq 0$) in Lemma \ref{L:diagHankel} together with Lemma \ref{L:limit} and the fact that $c_{m}=\int_{[0,1)^n}t_1^{2m_1}\cdots t_n^{2m_n}\polyint{n}{t}$, we conclude that
\begin{align*}
&\int_{[0,1)^{j}}|f(t_1,\ldots,t_j,1,\ldots,1)|^2t_1^{2m_1}\cdots t_j^{2m_j}\polyint{j}{t}\\
&\quad = \dfrac{\Big|\int_{[0,1)^{j}}f(t_1,\ldots,t_j,1,\ldots,1)t_1^{2m_1+s_1}\cdots t_j^{2m_j+s_j}\polyint{j}{t}\Big|^2}{\int_{[0,1)^{j}}t_1^{2m_1+2s_1}\cdots t_j^{2m_j+2s_j}\polyint{j}{t}}.
\end{align*}
Let $F(t)=f(t_1,\ldots,t_j,1,\ldots,1)t_1^{m_1}\cdots t_j^{m_j}$ and $G(t)=t_1^{m_1+s_1}\cdots t_j^{m_j+s_j}$ for $t=(t_1,\ldots, t_j)\in [0,1)^j$. Then the above identity shows that 
\begin{align*}
& \Big\{\int_{[0,1)^j}|F(t)|^2\polyint{j}{t}\Big\}\Big\{\int_{[0,1)^n}G^2(t)\polyint{j}{t}\Big\}\\
&\quad\quad\quad = \Big|\int_{[0,1)^j}(FG)(t)\polyint{j}{t}\Big|^2.
\end{align*}
This means that Holder's inequality applied to $F$ and $G$ is in fact an equality. Therefore, we have $F(t)=\alpha G(t)$, or equivalently,
\begin{equation}\label{Eq:const-func}
f(t_1,\ldots,t_j,1,\ldots,1)t_1^{m_1}\cdots t_j^{m_j}=\alpha t_1^{m_1+s_1}\cdots t_j^{m_j+s_j}
\end{equation}
for $\mu_1\times\cdots\times\mu_j$-a.e. $t=(t_1,\ldots,t_j)\in [0,1)^j$, where $\alpha$ is a constant. Since $1$ belongs to the support of all the measures $\mu_1,\ldots,\mu_j$, we may let $t_1=\cdots=t_j=1$ to obtain $\alpha=f(1,\ldots,1)$. Recall that in \eqref{Eq:const-func}, $m_1=\max\{0,-s_1\},\ldots, m_j=\max\{0,-s_j\}$.

Suppose first $s_1,\ldots, s_n\geq 0$. Since $m_1=\cdots=m_j=0$, we obtain from \eqref{Eq:const-func} that $f(t_1,\ldots,t_j,1,\ldots,1)=f(1,\ldots,1)t_1^{s_1}\cdots t_j^{s_j}$ for $\mu_1\times\cdots\times\mu_j$-a.e. $t=(t_1,\ldots,t_j)\in [0,1)^j$. This together with \eqref{Eq:quasi-homo} implies that $f(z)=f(1,\ldots,1)z^{s}$ for $\gamma$-a.e. $z$ in $W=\mathbb{D}^j\times\mathbb{T}^{n-j}$, which is a part of the boundary $\partial\mathbb{D}^n$.

Now suppose $s_p<0$ for some $1\leq p\leq n$. We will show that in this case $f(1,\ldots,1)=0$. Without loss of generality, we may assume $p=1$. For all large positive integers $m_2,\ldots,m_n$, let $m=(0,m_2,\ldots,m_n)$ (the assumption that $n\geq 2$ is needed here). Since $m+s\nsucceq 0$,
\begin{align*}
\lambda_{m} & = \frac{\int_{[0,1)^n}|f(t_1,\ldots,t_n)|^2t_2^{2m_2}\cdots t_n^{2m_n}\mathrm{d}\mu_1(t_1)\cdots\mathrm{d}\mu_n(t_n)}{\int_{[0,1)^n}t_2^{2m_2}\cdots t_n^{2m_n}\mathrm{d}\mu_1(t_1),\cdots,\mathrm{d}\mu_n(t_n)}.
\end{align*}
Letting $(m_2,\ldots,m_n)\to (\infty,\ldots,\infty)$ and using Lemma \ref{L:limit} together with the fact that $\lambda_{m}\to 0$, we conclude that
\begin{align*}
0 = \int_{[0,1)}|f(t_1,1,\ldots,1)|^2\mathrm{d}\mu_1(t_1).
\end{align*}
This implies $f(t_1,1,\ldots,1)=0$ for $\mu_1$-a.e. $t_1$ on $[0,1)$. Since $1$ is in the support of $\mu_1$ and $f$ is continuous at the point $(1,\ldots,1)$, it follows that $f(1,\ldots,1)=0$.

Now \eqref{Eq:const-func} gives $f(t_1,\ldots,t_j,1,\ldots,1)t_1^{m_1}\ldots t_j^{m_j}=0$ for $\mu_1\times\cdots\times\mu_j$-a.e. $t=(t_1,\ldots,t_j)$ in $[0,1)^j$. For such $t$, if $t_1^{m_1}\cdots t_j^{m_j}\neq 0$, then we have $f(t_1,\ldots,t_j,1\ldots,1)=0$. Otherwise, there exists $1\leq i\leq j$ so that $t_i=0$ and $m_i>0$. But $m_j=\max\{0,-s_i\}$, so $s_i<0$. Since $t_i=0=t_i\zeta_i$ for any $|\zeta_i|=1$, we have
\begin{align*}
f(t_1,\ldots,t_i,\ldots,t_j,1,\ldots,1) & = f(t_1,\ldots,t_i\zeta_i,\ldots,t_j,1,\ldots,1)\\
& = \zeta_i^{s_i}f(t_1,\ldots,t_i,\ldots,t_j,1,\ldots,1).
\end{align*}
This implies $f(t_1,\ldots,t_j,1,\ldots,1)=0$ because $\zeta_i$ can be chosen so that $\zeta_i^{s_i}\neq 1$. Therefore, $f(t_1,\ldots,t_j,1,\ldots,1)=0$ for $\mu_1\times\cdots\times\mu_j$-a.e. $t=(t_1,\ldots,t_j)$ in $[0,1)^j$, which implies $f(z)=0$ for $\gamma$-a.e. $z\in W=\mathbb{D}^j\times\mathbb{T}^{n-j}$.

The same argument applies to other parts of $\partial\mathbb{D}^n$ which are different from $\mathbb{T}^n$. On $\mathbb{T}^n$, \eqref{Eq:quasi-homo} gives $f(\zeta)=\zeta^s f(1,\ldots,1)$. If $s_p<0$ for some $1\leq p\leq n$ then since $f(1,\ldots,1)=0$ as showed above, we conclude that $f(\zeta)=0$ for $\zeta\in\mathbb{T}^n$. So the conclusions of the proposition also hold for $z$ in $\mathbb{T}^n\subset\partial\mathbb{D}^n$. The proof of the proposition is now completed.
\end{proof}

\section{Compact Hankel operators with more general symbols}\label{S:compact}

We have seen that any $f$ in $L^2_{\vartheta}$ admits the decomposition $f=\sum_{l\in\mathbb{Z}^n}f_l$, where $f_l=Q_l(f)$ is the orthogonal projection of $f$ on the space of quasi-homogeneous functions of multi-degree $l$. The next proposition shows that the compactness of $H_{f}$ implies the compactness of each $H_{f_l}$. We are then able to apply the results in the previous section. The Hardy space version of the proposition was proved in \cite{AhernJOT2009}. Our proof here, which also works for the Hardy space, is more direct.

\begin{proposition}\label{P:cptReduction}
Suppose $f\in L^2_{\vartheta}$ so that the operator $H_{f}$ is compact on $A^2_{\vartheta}$. Then for each $s\in\mathbb{Z}^n$, the operator $H_{f_s}$ is compact.
\end{proposition}

\begin{proof} For any $m\in\mathbb{Z}_{+}^{n}$, we have
\begin{align*}
H_{f}e_m & = (I-P)(fe_m) = (I-P)((\sum_{l\in\mathbb{Z}^n}f_l)e_m)\\
& = \sum_{l\in\mathbb{Z}^n}(I-P)(f_le_m) = \sum_{l\in\mathbb{Z}^n}H_{f_l}e_m.
\end{align*}
Since $Q_{s+m}(H_{f_l}e_m)=0$ if  $l\neq s$ and $Q_{s+m}(H_{f_s}e_m)=H_{f_s}e_m$, we obtain $Q_{s+m}(H_{f}e_m)=H_{f_s}e_m$. From Lemma \ref{L:diagHankel}, $H^{*}_{f_s}H_{f_s}$ is a diagonal operator with eigenvalues $\lambda_{m}$ given by
\begin{align*}
\lambda_{m} & = \|H_{f_s}e_m\|_{2}^2 = \|Q_{s+m}(H_{f}e_m)\|_{2}^2\leq \|H_{f}e_m\|_{2}^2.
\end{align*}
Since $H_{f}$ is compact, we have $\lim_{|m|\to\infty}\|H_{f}e_m\|_{2}=0$ (here $|m|=m_1+\cdots+m_n$). This implies $\lim_{|m|\to\infty}\lambda_m=0$ and hence, $H^{*}_{f_s}H_{f_s}$ is compact. Thus, $H_{f_s}$ is a compact operator.
\end{proof}

Suppose $g$ belongs to $A^2_{\vartheta}$. It was showed by Axler \cite{AxlerDMJ1986} that when $n=1$ and $\vartheta$ is the Lebesgue measure on the disk $\mathbb{D}$, $H_{\overline{g}}$ is compact if and only if $g$ is in the little Bloch space, that is, $\lim_{|z|\uparrow 1}(1-|z|^2)g'(z)=0$. For $n\geq 2$ and $\vartheta$ is the Lebesgue measure on the polydisc $\mathbb{D}^n$, a special case of \cite[Theorem D]{BekolleJFA1990} gives that $H_{\overline{g}}$ is compact if and only if $g$ is a constant function and in this case, $H_{\overline{g}}=0$. The following corollary to Proposition \ref{P:cptReduction} shows that this holds true for general measures $\vartheta$.

\begin{corollary}\label{C:holo-cpt}
Suppose $n\geq 2$. Let $g$ be a function in $A^2_{\vartheta}$ so that $H_{\overline{g}}$ is compact. Then $g$ is a constant function and hence $H_{\bar{g}}=0$.
\end{corollary}

\begin{proof} Write $g = \sum_{m\in\mathbb{Z}_{+}^n}c_{m} e_m$. For $m\in\mathbb{Z}_{+}^n$, since $Q_{-m}(\overline{g})=\bar{c}_{m}\bar{e}_{m}$ and $H_{\overline{g}}$ is compact, Proposition \ref{P:cptReduction} implies that $H_{\bar{c}_m\bar{e}_m}$ is compact. Theorem \ref{T:cpt-quasi} then shows that $\bar{c}_{m}=0$ for all $m\neq 0$. Thus, $g$ is a constant function.
\end{proof}

We are now ready for our main theorem in this paper.
\begin{theorem}\label{T:cpt-cont}
Suppose $n\geq 2$. Let $f$ be continuous on $\cldisk^n$ so that $H_{f}$ is a compact operator. Then there is a function $h$ which is continuous on $\cldisk^n$ and holomorphic on $\opdisk^n$, and a bounded function $g$ satisfying $\lim_{z\to\bdry}g(z)=0$ so that $f(z)=h(z)+g(z)$ for $\vartheta$-a.e. $z$ in $\mathbb{D}^n$.
\end{theorem}

\begin{proof} For any $s\in\mathbb{Z}^n$, Proposition \ref{P:cptReduction} shows that $H_{f_s}$ is compact. Since $f$ is continuous, each $f_s$ is also continuous. By Theorem \ref{T:cpt-quasi}, there is a holomorphic monomial $h_s$ so that $(f_s-h_s)(w)=0$ for $\gamma$-a.e. $w\in\partial\mathbb{D}^n$. (In fact, $h_s=0$ if $s\nsucceq 0$ and $h_s(w)=f_s(1,\ldots,1)w^s$ if $s\succeq 0$).  

For each integer $N\geq 1$, define $$p_N(z)=\sum_{|s_1|,\ldots,|s_n|\leq N}\big(1-\frac{|s_1|}{N+1}\big)\cdots\big(1-\frac{|s_n|}{N+1}\big)h_{s_1,\ldots,s_n}(z)\quad\text{for } z\in\cldisk^n.$$
Then $p_N$ is a holomorphic polynomial and $p_N(w)=\Lambda_{N}(f)(w)$ for $\gamma$-a.e. $w\in\partial\mathbb{D}^n$, 
where $\Lambda_{N}(f)$ is the $N$th Ces{\`a}ro mean of $f$. 
Since $\gamma$ restricted on $\mathbb{T}^n\subset\partial\cldisk^n$ is the surface measure and $p_N-\Lambda_N(f)$ is continuous, 
we actually have $p_N(w)=\Lambda_N(f)(w)$ for all $w\in\mathbb{T}^n$. 
By the remark at the end of Section 2, $\Lambda_N(f)$ converges to $f$ uniformly on $\cldisk^n$. In particular,  $p_N|_{\mathbb{T}^n}=\Lambda_N(f)|_{\mathbb{T}^n}$ converges to $f|_{\mathbb{T}^n}$ uniformly. 
This implies that there is a function $h$ which is continuous on $\cldisk^n$ and holomorphic on $\opdisk^n$ so that $p_N$ converges uniformly to $h$ on $\cldisk^n$. Since $p_N(w)=\Lambda_N(f)(w)$ for $\gamma$-a.e. $w\in\partial\opdisk^n$ for each $N$, we conclude that $h(w)=f(w)$ for $\gamma$-a.e. $w\in\partial\opdisk^n$. Let $\tilde{g}=f-h$. Then $\tilde{g}$ is continuous on $\cldisk^n$ and $\tilde{g}(w)=0$ for $\gamma$-a.e. $w$ on $\partial\opdisk^n$. By Lemma 2.1 in \cite{LePAMS2009}, there is a function $g$ such that $\tilde{g}(z)=g(z)$ for $\vartheta$-a.e. $z$ in $\mathbb{D}^n$ and $\lim_{z\to w}g(z)=0$ for all $w\in\partial\opdisk^n$. From the compactness of $\bdry$, it follows that $\lim_{z\to\bdry}g(z)=0$. We then have $f(z)=h(z)+\tilde{g}(z)=h(z)+g(z)$ for $\vartheta$-a.e. $z\in\mathbb{D}^n$ and $h,g$ satisfy the requirements in the conclusion of the theorem.
\end{proof}

The continuity of $f$ on $\cldisk^n$ in Theorem \ref{T:cpt-cont} cannot be dropped. Indeed, there are bounded functions $f$ which are continuous on the open polydisc $\opdisk^n$ such that $H_{f}$ is compact and no decomposition $f=h+g$ with $h\in A^2_{\vartheta}$ and $\lim_{z\to\partial\mathbb{D}^n}g(z)=0$ is possible. In the rest of the section, we will give a construction of such a function.

Let $0<r_1<r_2<\cdots$ be an increasing sequence of positive real numbers that converges to $1$. Set $j_1=1$. Since $(r_{j_1},1)^n = \cup_{j=j_1+1}^{\infty}(r_{j_1},r_j)^n$ and $\mu((r_{j_1},1)^n)>0$, there is an integer $j_2\geq j_1+1$ so that $\mu((r_{j_1},r_{j_2})^n)>0$. Since $(r_{j_2},1)^n=\cup_{j=j_2+1}^{\infty}(r_{j_2},r_{j})^n$ and $\mu((r_{j_2+1},1)^n)>0$, there is an integer $j_3\geq j_2+1$ so that $\mu((r_{j_2},r_{j_3})^n)>0$. Continuing this process, we find a sequence of integers $\{j_k\}_{k=1}^{\infty}$ such that $j_{k+1}\geq j_k+1$ and $\mu((r_{j_{k}},r_{j_{k+1}})^n)>0$ for $k=1,2,\ldots$. For each such $k$, let $R_{k}=(r_{j_k},r_{j_{k+1}})^n$. Choose an open subset $V_{k}$ of $\mathbb{T}^n$ so that $0<\sigma(V_{k})<1/k$ and $\int_{E_k}K_z(z)\mathrm{d}\vartheta(z)<1/k^2$, where $E_{k}=\{r\cdot\zeta: r\in R_{k}, \zeta\in V_{k}\}$. The existence of $V_k$ follows from the fact that the function $z\mapsto K_z(z)$ is bounded on compact sets. Since $E_k$ is open in $\opdisk^n$ and $\vartheta(E_k)>0$, using the regularity of $\vartheta$, we can choose a continuous function $0\leq f_k\leq 1$ so that $f_k$ is supported in $E_{k}$ and $\vartheta(\{z: f(z)=1\})>0$. Put $f=\sum_{k=1}^{\infty}f_k$. Since the sets $E_k$'s are pairwise disjoint, the function $f$ is continuous on $\opdisk^n$ and $0\leq f(z)\leq 1$ for all $z\in\opdisk^n$. 

We now show that $H_{f}$ is a Hilbert-Schmidt operator, hence it is compact. Indeed, from \eqref{E:Hilbert-Schmidt},
\begin{align*}
\sum_{m\in\mathbb{Z}_{+}^n}\|H_fe_m\|_{2}^2 & \leq \int_{\mathbb{D}^n}|f(z)|^2K_z(z)\mathrm{d}\vartheta(z)\\
&\leq\sum_{k=1}^{\infty}\int_{E_k}|f(z)|^2K_z(z)\mathrm{d}\vartheta(z)<\sum_{k=1}^{\infty}\frac{1}{k^2}<\infty.
\end{align*}

For each $s\in\mathbb{Z}^n$, we will show that $\lim_{z\to\bdry}Q_s(f)(z)=0$. Since $f\geq 0$, it follows from formula \eqref{Eq:Q_s} in Lemma \ref{L:decomposition} that $|Q_s(f)(z)|\leq |Q_{0}(f)(z)|$. So it suffices to prove $\lim_{z\to\bdry}Q_{0}(f)(z)=0$. For $z=(z_1,\ldots,z_n)\in\opdisk^n$, we have
\begin{align*}
Q_{0}(f)(z) & = \int_{\mathbb{T}^n}f(z\cdot\zeta)\mathrm{d}\sigma(\zeta) = \int_{\mathbb{T}^n}f(|z_1|\zeta_1,\ldots,|z_n|\zeta_n)\mathrm{d}\sigma(\zeta)\\
&\leq \sum_{k=1}^{\infty}\int_{\mathbb{T}^n}\chi_{E_k}(|z_1|\zeta_1,\ldots,|z_n|\zeta_n)\mathrm{d}\sigma(\zeta).
\end{align*}
By the definition of $E_k$, $\chi_{E_k}(|z_1|\zeta_1,\ldots,|z_n|\zeta_n)=1$ if and only if the $n$-tuple $(|z_1|,\ldots,|z_n|)$ belongs to $R_k=(r_{j_k},r_{j_{k+1}})^n$ and $(\zeta_1,\ldots,\zeta_n)$ belongs to $V_{k}$. Let $k_0\geq 2$ be a positive integer. Suppose $z=(z_1,\ldots,z_n)$ so that $|z_i|>r_{j_{k_0}}$ for some $1\leq i\leq n$. If $(|z_1|,\ldots,|z_n|)$ does not belong to any $R_k$, $k\geq 1$, then $Q_0(f)(z)=0$. Otherwise, there is exactly one $k$ so that $(|z_1|,\ldots,|z_n|)\in R_k$, which is $(r_{j_k},r_{j_{k+1}})^n$. Since $|z_i|>r_{j_{k_0}}$, we conclude that $r_{j_{k+1}}>r_{k_0}$, which implies $k\geq k_0$. Therefore,
\begin{align*}
|Q_0(f)(z)| &\leq\int_{\mathbb{T}^n}\chi_{E_k}(|z_1|\zeta_1,\ldots,|z_n|\zeta_n)\mathrm{d}\sigma(\zeta)=\sigma(V_{k})\leq\frac{1}{k}\leq\frac{1}{k_0}.
\end{align*}
Since this holds true for any $z$ that does not belong to the compact set $[0,r_{j_{k_0}}]^n\times\mathbb{T}^n$, we conclude that $\lim_{z\to\bdry}Q_0(f)(z)=0$.

Suppose there were a decomposition $f=h+g$, where $h\in A^2_{\vartheta}$ and $g(z)\to 0$ as $z\to\bdry$. We will show that there would be a contraction. For $s\succeq 0$, from formula \eqref{Eq:Q_s} in Lemma \ref{L:decomposition} we see that $Q_{s}(g)(z)\to 0$ as $z\to\bdry$. This implies $Q_s(h)(z)=Q_s(f)(z)-Q_s(g)(z)\to 0$ as $z\to\bdry$. But $Q_s(h)$ is a multiple of $z^{s}$, as explained at the beginning of Section 3, so we have $Q_s(h)=0$ for all $s\succeq 0$. It follows that $h=0$ and hence, $f=g$. This is a contradiction because $g(z)\to 0$ as $z\to\bdry$ but by the construction of $f$, for any compact subset $M\subset\opdisk^n$, the set $\{z\in\opdisk^n\backslash M: f(z)=1\}$ has positive $\vartheta$-measure.

\end{document}